\documentclass[12pt]{amsart}
\usepackage[hmargin=2.5cm,vmargin=2.5cm]{geometry}
\usepackage{amsfonts, amstext, amsmath, amsthm, amscd, amssymb}
\usepackage[pdftex]{graphicx, color}

\usepackage[noadjust]{cite}
\usepackage[hidelinks,pagebackref]{hyperref}
\usepackage{enumitem}
\usepackage{setspace}
\usepackage{float}
\usepackage{subcaption}
\usepackage{mathrsfs}
\usepackage{tikz}
\usepackage{pinlabel}
\usepackage{cleveref}
\usepackage{multirow}

\AtBeginDocument{
	\def\MR#1{}
}

\pagecolor{white}

\allowdisplaybreaks

\newcommand{\Z}{\mathbb{Z}}

\newcommand{\Q}{\mathbb{Q}}

\newcommand{\F}{\mathbb{F}}

\newtheorem{thm}{Theorem}
\numberwithin{thm}{section}

\newtheorem{conj}[thm]{Conjecture}

\newtheorem{lemma}[thm]{Lemma}
\newtheorem{cor}[thm]{Corollary}
\newtheorem{question}[thm]{Question}

\newtheorem{claim}[thm]{Claim}

\newtheorem*{namedtheorem}{\theoremname}
\newcommand{\theoremname}{testing}
\newenvironment{named_thm}[1]{\renewcommand{\theoremname}{#1}\begin{namedtheorem}}{\end{namedtheorem}}

\theoremstyle{definition}

\newtheorem*{nameddef}{\defname}
\newcommand{\defname}{testing}

\theoremstyle{definition}
\newtheorem{rmk}[thm]{Remark}

\begin{document}
	\title{An Alexander Polynomial Obstruction to Cosmetic Crossing Changes}
	\author{Joe Boninger}
	\address{Department of Mathematics, Boston College, Chestnut Hill, MA}
	\email{boninger@bc.edu}
	\maketitle
	
	\begin{abstract}
		The cosmetic crossing conjecture posits that switching a non-trivial crossing in a knot diagram always changes the knot type. Generalizing work of Balm, Friedl, Kalfagianni and Powell, and of Lidman and Moore, we give an Alexander polynomial condition that obstructs cosmetic crossing changes for knots with $L$-space branched double covers, a family that includes all alternating knots. As an application, we prove the cosmetic crossing conjecture for a five-parameter infinite family of pretzel knots. We also discuss the status of the conjecture for alternating knots with eleven crossings.
	\end{abstract}
	
	\section{Introduction}
	
	The cosmetic crossing conjecture, attributed to X.~S.~Lin and appearing on Kirby's famous problem list \cite[Problem 1.58]{kir78}, posits that switching a non-trivial crossing in an oriented knot diagram always changes the knot type. More precisely, let $K \subset S^3$ be an oriented knot and $D \subset S^3$ an oriented disk intersecting $K$ in two points of opposite sign; then a {\em crossing change} on $K$ is the operation of performing $\pm 1$-surgery on the unknot $\partial D$. We say a crossing change is {\em cosmetic} if the resulting oriented knot is equivalent to $K$, and the change is {\em nugatory} if $\partial D$ bounds a disk in $S^3 - K$. We also define a {\em generalized crossing change} to be the result of $1/n$ surgery on $\partial D$ for some $n \neq 0$, with cosmetic and nugatory generalized crossing changes defined as before. It's not hard to see that any nugatory crossing change is cosmetic; the cosmetic crossing conjecture asserts the converse is also true.
	
	\begin{conj}[Cosmetic Crossing Conjecture]
		\label{conj:ccc}
		Any cosmetic crossing change is nugatory.
	\end{conj}
	Some authors (e.g.~\cite{bal15, wan22}) have also considered a ``generalized'' cosmetic crossing conjecture, with the goal of showing any cosmetic generalized crossing change is nugatory.
	
	In recent years \Cref{conj:ccc} has seen a flurry of activity---it is now known that knots in the following families to do not admit cosmetic, non-nugatory crossing changes:
	\begin{itemize}
		\item Two-bridge knots, due to Torisu \cite{tor99}.
		\item Fibered knots, due to Kalfagianni \cite{kal12}.
		\item Genus one knots with non-trivial Alexander polynomial, due to Ito following work of Balm, Friedl, Kalfagianni and Powell \cite{ito222, bfkp12}.
		\item Alternating knots with square-free determinant, due to Lidman and Moore \cite{lm17}.
		\item Special alternating knots, due to the author \cite{bon23_3}.
	\end{itemize}
	Lidman and Moore also give a stronger obstruction using the homology of the branched double cover \cite[Theorem 2]{lm17}, and see \cite{bk16, ito22, ito223, moore16} for more relevant work.
	
	In this note we give an Alexander polynomial obstruction to cosmetic crossing changes that generalizes two of the above results. Given any knot $K \subset S^3$, let $\Sigma(K)$ denote its branched double cover. The three-manifold $\Sigma(K)$ is an {\em $L$-space} if it satisfies a Heegaard-Floer theoretic condition---see \cite{os5b} for details.
	
	\begin{thm}
		\label{thm:main}
		Let $K \subset S^3$ be a knot such that $\Sigma(K)$ is an $L$-space. If $K$ admits a non-nugatory, cosmetic generalized crossing change, then its Alexander polynomial $\Delta_K$ has a factor of the form
		$$
		f(t)f(t^{-1})
		$$
		for some $f \in \Z[t, t^{-1}]$ satisfying $f(-1) \neq \pm 1$. In particular, $f$ is non-constant.
	\end{thm}

	The fact that $f$ is non-constant in \Cref{thm:main} follows from the identity $\Delta_K(1) = \pm 1$, so that $f(1) = \pm 1$. Additionally, the branched double cover $\Sigma(K)$ of a knot $K \subset S^3$ is an $L$-space if $K$ is an alternating knot, or more generally if $K$ has thin Khovanov homology with $\Z/2$ coefficients \cite{os05}. We thus obtain:
	
	\begin{cor}
		Let $K \subset S^3$ be an alternating knot. If $K$ admits a non-nugatory, cosmetic crossing change, then $\Delta_K$ has a factor of the form $f(t)f(t^{-1})$ for some $f \in \Z[t,t^{-1}]$ satisfying $f(-1) \neq \pm 1$.
	\end{cor}
	
	Balm, Friedl, Kalfagianni and Powell prove a statement analogous to \Cref{thm:main} for genus one knots \cite[Theorem 1.1(1)]{bfkp12}: their theorem does not require that $\Sigma(K)$ be an $L$-space, but allows for the possibility that $\Delta_K \equiv f \equiv 1$. (By Ito's work, this is the only remaining case open for genus one knots \cite{ito222}.) \Cref{thm:main} also recovers the aforementioned result of Lidman and Moore.
	
	\begin{cor}[\protect{\cite[Corollary 3]{lm17}}]
		\label{cor:lm}
		Let $K \subset S^3$ be a knot such that $\Sigma(K)$ is an $L$-space and $K$ admits a non-nugatory, cosmetic crossing change. Then the knot determinant $\det(K)$ has a square factor.
	\end{cor}
	
	\begin{proof}
		This is immediate from \Cref{thm:main} and the identity $\det(K) = \Delta_K(-1)$.
	\end{proof}

	\Cref{thm:main} is a good obstruction to cosmetic crossing changes for knots with low crossing number, and alternating knots in particular---it verifies \Cref{conj:ccc} for 200 of the 250 prime knots with 10 or fewer crossings, and 483 of the 564 prime alternating knots with 11 crossings or fewer \cite{knotinfo}. Thanks to the authors referenced above, the cosmetic crossing conjecture has already been proven for all but four knots with 10 or fewer crossings \cite[Corollary 1.11]{ito223}\cite{bon23_3}, but \Cref{thm:main} proves the conjecture for new knots as well. In \Cref{sec:eleven} below we examine the case of eleven-crossing alternating knots in detail, and in \Cref{sec:family} we use \Cref{thm:main} to prove:
	
	\begin{thm}
		\label{thm:example}
		Let $p_1$, $p_2$, $p_3$, $p_4$ and $q$ be integers satisfying:
		\begin{enumerate}[label=(\roman*)]
			\item $p_i \geq 1$ for all $i$, and $q > \min(p_1, \dots, p_4)$.
			\item $p_i \equiv 1$ modulo $4$ for all $i$, and $q \equiv 3$ modulo $4$.
		\end{enumerate}
		Then the pretzel knot $P(p_1,p_2,p_3,p_4,-q)$ does not admit a non-nugatory, cosmetic generalized crossing change.
	\end{thm}

	While some of the knots in \Cref{thm:example} are already known to satisfy \Cref{conj:ccc}, infinitely many are not; see our discussion in \Cref{sec:example}. Finally, our \Cref{thm:main} gives a simplified proof of \Cref{conj:ccc} for the knots $10_{108}$ and $10_{164}$, which Ito handled using the Casson-Walker invariant and the calculus of Jacobi diagrams \cite{ito223}.
	
	\subsection{Discussion}
	
	For a genus one knot $K$, the conclusion of \Cref{thm:main} implies $K$ is algebraically concordant to the unknot---see \cite{bfkp12, lev69}. Conversely, knots which are concordant to knots with strictly smaller genus frequently have factors of the form $f(t)f(t^{-1})$ in their Alexander polynomials (cf.~\cite{gil84}). We therefore pose the following question as a next step in studying \Cref{conj:ccc}.
	
	\begin{question}
		If a knot $K$ admits a non-nugatory cosmetic crossing change, is $K$ concordant to a knot with strictly smaller genus than itself?
	\end{question}
	
	\subsection{Outline}
	
	In \Cref{sec:top} we collect useful facts, and in \Cref{sec:proof} we prove \Cref{thm:main}. In \Cref{sec:example} we prove \Cref{thm:example} and we discuss the state of \Cref{conj:ccc} for knots with low crossing number.
	
	\subsection{Acknowledgments}
	
	The author thanks Josh Greene for mathematical and professional guidance, and Ali Naseri Sadr for many worthwhile discussions about the cosmetic crossing conjecture. The author is also grateful to John Baldwin, Keerthi Madapusi and Ian Montague for helpful conversations, and Brendan Owens for feedback on an earlier draft of this paper. This material is based upon work supported by the National Science Foundation under Award No.~2202704.
	
	\section{Topological Set-Up}
	\label{sec:top}
	
	Let $K \subset S^3$ be an oriented knot. When considering a crossing change on $K$ given by a disk $D$, as described in the introduction, we refer to $D$ as a {\em crossing disk} associated to the change. Similarly, a {\em crossing arc} is an arc in $D$ connecting the two points of $K \cap D$. Finally, let $\pi : \Sigma(K) \to S^3$ be the double cover of $S^3$ branched along $K$, and let $\pi^{-1}$ indicate taking preimages in $\pi$. If $\alpha \subset S^3$ is a crossing arc for $K$, then $\pi^{-1}(\alpha) \subset \Sigma(K)$ is a simple closed curve.
	
	A key tool in our proof is the following result of Lidman and Moore.
	
	\begin{thm}[{\cite[Remark 13]{lm17}}]
		\label{thm:lm}
		Suppose $\Sigma(K)$ is an $L$-space and $\alpha \subset S^3$ is the crossing arc of a cosmetic generalized crossing change for $K$.  If the curve $\pi^{-1}(\alpha) \subset \Sigma(K)$ is nullhomologous in $H_1(\Sigma(K))$, then the crossing change is nugatory.
	\end{thm}

	Lidman and Moore's theorem is stated only for standard crossing changes, but their proof works for generalized crossing changes as well. An important ingredient is the surgery characterization of the unknot among nullhomologous knots in an $L$-space, due to Gainullin \cite{gai18}.

	The goal of the next two lemmas is to give a criterion for applying \Cref{thm:lm}. Let $S \subset S^3$ be a Seifert surface for $K$, and let $\alpha \subset S$ be a properly embedded, oriented, non-separating arc. Let $S' \subset S^3$ be the surface $S - \nu(\alpha)$, where $\nu$ denotes a regular open neighborhood; we identify $H_1(S')$ with a subgroup of $H_1(S)$ via the inclusion $S' \hookrightarrow S$.

	\begin{lemma}
		\label{lem:basis}
		There exists a basis $\mathcal{B} = [e_1, \dots, e_{2g}]$ for $H_1(S)$ such that:
		\begin{align*}
			e_1 \cdot \alpha &= 1 \\
			e_i \cdot \alpha &= 0 \text{ for all } i \neq 1,
		\end{align*}
		where $\cdot$ denotes the intersection pairing. It follows that $[e_2, \dots, e_{2g}]$ is a basis for $H_1(S')$.
	\end{lemma}

	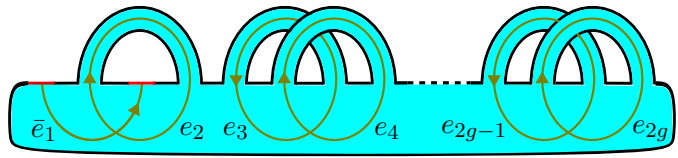
\begin{figure}
		\labellist
		\small\hair 2pt
		\pinlabel $\bar{e}_1$ at 70 50
		\pinlabel $e_2$ at 360 50
		\pinlabel $e_3$ at 445 50
		\pinlabel $e_4$ at 740 50
		\pinlabel $e_{2g - 1}$ at 910 50
		\pinlabel $e_{2g}$ at 1255 50
		\endlabellist
		
		\includegraphics[height=2cm]{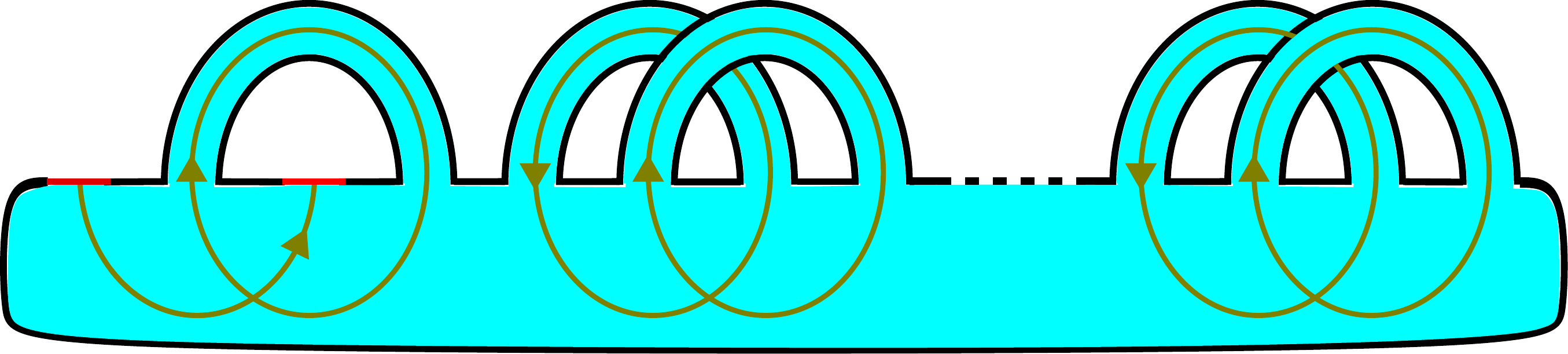}
		
		\caption{A basis for $H_1(S)$}
		\label{fig:hombasis}
	\end{figure}

	\begin{proof}
		By the classification of surfaces with boundary, $S'$ is homeomorphic to the surface in \Cref{fig:hombasis}. Additionally, by an isotopy supported in a small neighborhood of $\partial S'$, we assume the two arcs $\partial \nu(\alpha) \subset \partial S'$ are the red arcs shown in the figure. We then choose the classes $e_2, \dots, e_{2g}$ as pictured in the figure, and let $e_1$ be the class of the curve resulting from joining the ends of the arc $\bar{e}_1$ through $\nu(\alpha)$ in $S$.
	\end{proof}

	\begin{lemma}
		\label{lem:two}
		Let $\mathcal{B}$ be a basis for $H_1(S)$ as given by \Cref{lem:basis}, and let $V$ be the Seifert matrix of $S$ in the basis $\mathcal{B}$. Then the matrix $V + V^\intercal$ is a presentation matrix for $H_1(\Sigma(K))$ with ordered generating set $b_1, \dots, b_{2g}$, such that
		$$
		b_1 = 2[\pi^{-1}(\alpha)] \in H_1(\Sigma(K))
		$$
		for some orientation of the curve $\pi^{-1} (\alpha)$.
	\end{lemma}

	The fact that $V + V^\intercal$ presents $H_1(\Sigma(K))$ is well known---the substance of \Cref{lem:two} is the claim that $b_1 = 2[\pi^{-1}(\alpha)]$.
	
	\begin{figure}
		\labellist
		\small\hair 2pt
		\pinlabel $\mathring{S}^+$ at 245 290
		\pinlabel $\mathring{S}^-$ at 245 165
		\pinlabel $U$ at 430 405
		\pinlabel $U_1$ at 1280 405
		\pinlabel $\mathring{\alpha}_1$ at 1110 355
		\pinlabel $\mathring{\alpha}_2$ at 1110 85
		\pinlabel $U_2$ at 1085 220
		\pinlabel $D$ at 150 350
		\endlabellist
		\subcaptionbox{The curve $U$ in $Y$ \label{fig:brancha}}{
			\includegraphics[height=3cm]{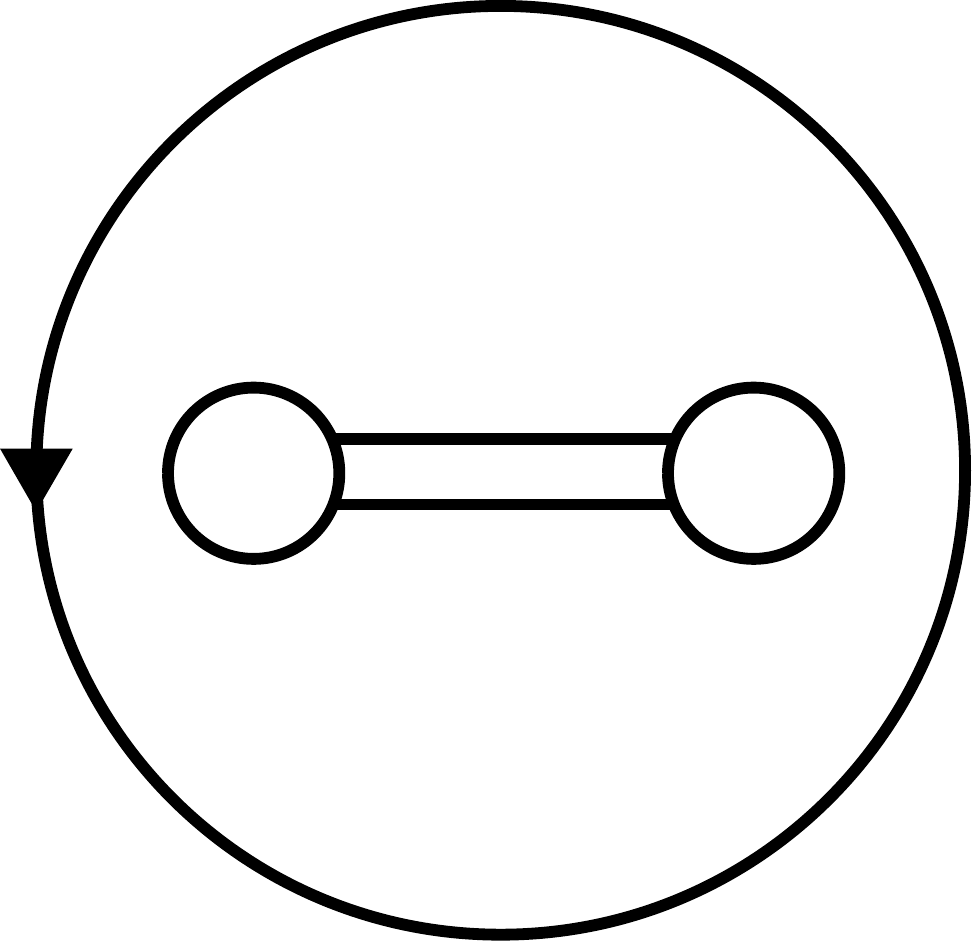}
		}
		\hspace{2cm}
		\subcaptionbox{The $U_i$ curves in $\Sigma(K)$ \label{fig:branchb}}{
			\includegraphics[height=3cm]{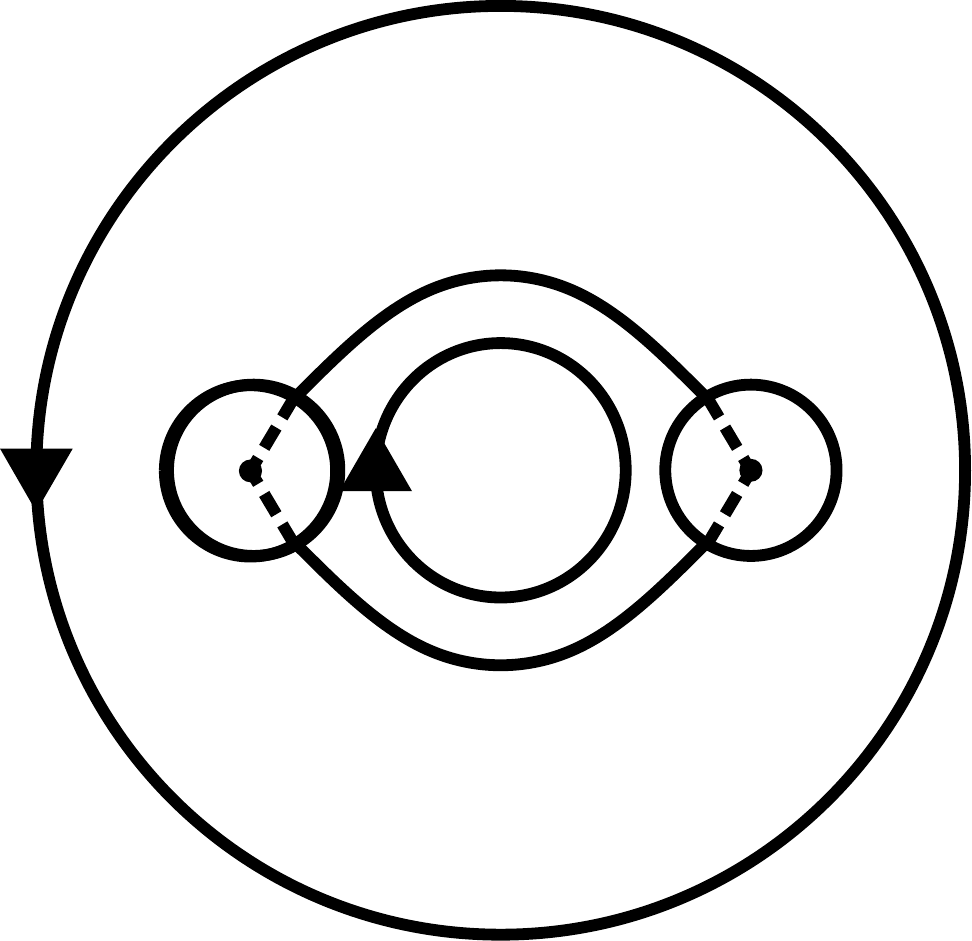}
		}
		\caption{}
	\end{figure}
	
	\begin{proof}
		Let $\mathcal{B} = e_1, \dots, e_{2g} \in H_1(S)$ be as in the lemma, and let $f_1, \dots, f_{2g}$ be the dual basis of $H_1(S^3 - S)$ given by the Alexander duality isomorphism
		$$
		H_1(S^3 - S) \cong H^1(S) \cong H^*_1(S).
		$$
		Let $D \subset S^3$ be an oriented disk with $D \cap S = \alpha$, and let $U = \partial D$. By \Cref{lem:basis} the linking pairing
		$$
		\text{lk}(*, [U]) : H_1(S) \to \Z
		$$
		satisfies
		$$
		\text{lk}(e_i, [U]) = e_i \cdot \alpha = \delta_{i1} = \text{lk}(e_i, f_1),
		$$
		where $\delta_{i1}$ is the Kronecker delta, so $[U] = f_1 \in H_1(S^3 - S)$ by Alexander duality.
		
		Since $U \cap S = \varnothing$, $\text{lk}([U], [K]) = 0$ and $\pi^{-1}(U) \subset \Sigma(K)$ has two components, which we denote by $U_1$ and $U_2$. We claim that, for some choice of orientation of $\pi^{-1}(\alpha)$,
		\begin{equation}
			\label{eq:claim}
			[U_1] = [\pi^{-1}(\alpha)] = -[U_2] \in H_1(\Sigma(K)).
		\end{equation}
		To this end, let $X = S^3 - \nu(K)$ be the exterior of $K$, let $\mathring{S} = S \cap X$, and let $Y = X - \nu(\mathring{S})$. We choose our regular neighborhoods of $K$ and $S$ small enough that 
		$$
		U \cap \nu(K) = U \cap \nu(S) = \varnothing,
		$$
		so that $U \subset Y$ and $D \cap \partial Y$ is a circle $U'$ parallel to $U$ in $D \cap Y$. The boundary $\partial \nu(\mathring{S}) \subset \partial Y$ has two components, both parallel to $\mathring{S}$ in $X$, which we label $\mathring{S}^+$ and $\mathring{S}^-$ according to the orientation of $S$. See \Cref{fig:brancha}.
		
		Let $Y_1$ and $Y_2$ be two copies of $Y$, and let $\mathring{S}^\pm_i$ be the copy of $\mathring{S}^\pm$ in $Y_i$ for $i = 1, 2$. Then the cyclic double cover $\bar{X}_2$ of $X$ can be constructed by gluing $\mathring{S}^+_1$ to $\mathring{S}^-_2$ and $\mathring{S}^+_2$ to $\mathring{S}^-_1$:
		$$
		\bar{X}_2 = Y_1 \cup_{\mathring{S}^+_i = \mathring{S}^-_{3 - i}} Y_2.
		$$
		The branched double cover $\Sigma(K)$ can be obtained from $\bar{X}_2$ by filling its torus boundary with a solid torus $T$ along the natural choice of meridian---see \cite{l97} for details. The curves $U_1, U_2 \subset \Sigma(K)$ that make up $\pi^{-1}(U)$ are precisely the copies of $U$ in $Y_1$ and $Y_2$, and similarly we let $U'_i$ be the copy of $U'$ in $Y_i$, so that
		$$
		\pi^{-1}(U') = U'_1 \cup U'_2.
		$$
		Let $\mathring{\alpha} = \alpha \cap \mathring{S}$; then the preimage of $\mathring{\alpha}$ in $\bar{X}_2$ is the two arcs $\mathring{\alpha}_1$ and $\mathring{\alpha}_2$ which are the push-offs of $\mathring\alpha$ to $\mathring{S}^+_1 (= \mathring{S}^-_2)$ and $\mathring{S}^+_2 (= \mathring{S}^-_1)$ respectively. The preimage $\pi^{-1}(\alpha) \subset \Sigma(K)$ is therefore a closed curve obtained by connecting $\mathring\alpha_1$ and $\mathring\alpha_2$ to the core of $T$ via four radial arcs through two meridian disks, as in \Cref{fig:branchb}. Forgetting orientations, this curve is clearly parallel to each of $U'_1$ and $U'_2$ in $\Sigma(K)$, hence parallel to each of $U_1$ and $U_2$. Fixing an orientation for $\pi^{-1}(\alpha)$ without loss of generality, we have
		$$
		[U_1] = [\pi^{-1}(\alpha)] = -[U_2],
		$$
		proving (\ref{eq:claim}).
		
		Following the notation of the previous paragraph, let $f_1^1$ and $f_1^2$ be the two copies of $f_1$ in $H_1(Y_1)$ and $H_1(Y_2)$. A careful reading of the proof of \cite[Theorem 9.1]{l97} now shows the matrix $V + V^\intercal$ presents the group $H_1(\Sigma(K))$, with generating set $b_1, \dots, b_{2g}$ such that
		$$
		b_1 = f_1^1 - f_1^2 = [U_1] - [U_2] = 2[\pi^{-1}(\alpha)] \in H_1(\Sigma(K)).
		$$
	\end{proof}

	\begin{cor}
		\label{cor:helpful}
		Let $\mathcal{B}$ be a basis for $H_1(S)$ given by \Cref{lem:basis}, and let $V$ be the Seifert matrix of $S$ in the basis $\mathcal{B}$. Let $b_1$ be the first standard basis vector of $\Z^{2g}$. Then $\pi^{-1}(\alpha)$ is nullhomologous if and only if $b_1$ is in the image of the map
		$$
		V + V^\intercal : \Z^{2g} \to \Z^{2g}.
		$$
	\end{cor}

	\begin{proof}
		By \Cref{lem:two}, $b_1$ is in the image of $V + V^\intercal$ if and only if $2[\pi^{-1}(\alpha)] \in H_1(\Sigma(K))$ is nullhomologous. Since $H_1(\Sigma(K))$ is a finite group of odd order (equal to the determinant of $K$), $2[\pi^{-1}(\alpha)]$ is nullhomologous if and only if $[\pi^{-1}(\alpha)]$ is.
	\end{proof}

	Along with \Cref{thm:lm}, our proof of \Cref{thm:main} relies on the following simple observation.
	
	\begin{lemma}[\protect{\cite[Proposition 1.2]{wan22}}]
		\label{lem:alex}
		Let $L$ be the two-component link $\partial S'$. If $\alpha$ is the crossing arc of a cosmetic generalized crossing change whose crossing disk $D$ satisfies $D \cap S = \alpha$, then $\Delta_L \equiv 0$.
	\end{lemma}

	\begin{proof}
		Suppose the crossing change is $1/n$ surgery on $\partial D$ for some $n  > 0$, and let $K_n \cong K$ be the resulting knot. If $n = 1$, then without loss of generality the Alexander polynomials of $K$, $K_1$ and $L$ satisfy the skein relation
		\begin{equation}
			\label{eq:alexone}
			\Delta_{K_1}(t) - \Delta_{K}(t) = (t^{1/2} - t^{-1/2})\Delta_L(t).
		\end{equation}
		More generally, for all $j > 0$ we have
		\begin{equation}
			\label{eq:alextwo}
			\Delta_{K_{j + 1}}(t) - \Delta_{K_j}(t) = (t^{1/2} - t^{-1/2})\Delta_L(t).
		\end{equation}
		Summing (\ref{eq:alexone}) with $n - 1$ copies of (\ref{eq:alextwo}), where $j$ ranges from $1$ to $n - 1$, gives
		$$
		\Delta_{K_n}(t) - \Delta_{K}(t) = n(t^{1/2} - t^{-1/2})\Delta_L(t).
		$$
		Since $K_n = K$ by hypothesis, $\Delta_{K_n} = \Delta_K$ and $\Delta_L \equiv 0$.
	\end{proof}

	Our proof of \Cref{thm:main} in the next section contains a few subtleties, but the underlying idea is simple: let $K$ be a knot with $\Sigma(K)$ an $L$-space, such that $K$ admits a cosmetic, non-nugatory generalized crossing change. Then we use \Cref{lem:alex} and the symmetry of the Alexander matrix to show $\Delta_K$ admits a factor of the form $f(t)f(t^{-1})$, and we use \Cref{cor:helpful} and \Cref{thm:lm} to show this factor is non-trivial.
	
	\section{Proof of \protect\Cref{thm:main}}
	\label{sec:proof}

		Let $K$ be a knot with $\Sigma(K)$ an $L$-space, and let $D \subset S^3$ be a crossing disk associated to a cosmetic generalized crossing change on $K$. Since the linking number of $K$ with $\partial D$ is zero, there exists a Seifert surface $S \subset S^3$ for $K$ such that $S \cap \partial D = \varnothing$. In fact, any Seifert surface can be modified to satisfy this condition by adding tubes along pairs of oppositely signed points of $S \cap \partial D$, beginning with an innermost pair. Additionally, by shrinking $D$ if necessary, we assume $S \cap D$ consists of a single crossing arc $\alpha$. We also assume $\alpha$ does not separate $S$: if this is not the case, we change $S$ by adding a tube connecting the components of $S - \alpha$ and disjoint from $D$. Let $g$ be the genus of $S$, and let $S' \subset S^3$ be the surface $S - \nu(\alpha)$ as above. Let $L$ be the link $\partial S'$.
						
		 Let $\mathcal{B}$ be a basis for $H_1(S)$ given by \Cref{lem:basis}, using the crossing arc $\alpha$ as the relevant non-separating arc. Let $V$ be the Seifert matrix of $S$ in the basis $\mathcal{B}$, and let $A$ be the Alexander matrix $V - tV^\intercal$. Let $A'$ be the $(2g - 1)$-by-$(2g - 1)$ matrix given by removing the first row and column of $A$; then $A'$ is an Alexander matrix for $L$ by \Cref{lem:basis}, and by \Cref{lem:alex} we have
		 $$
		 0 = \Delta_L(t) = \det(A').
		 $$
		 Let $v_1 = [x_1, \dots, x_{2g - 1}]^\intercal \in \Q(t)^{2g - 1}$ be a vector in the kernel of $A'$. Multiplying $v_1$ by the least common denominator of the $x_i$, we assume $v_1$ is in the subring $\Q[t]^{2g - 1}$, and multiplying again by the least common denominator of the coefficients of all the $x_i$, we further assume that $v_1 \in \Z[t]^{2g - 1}$. Finally, after dividing out any common factors, we have $\gcd(x_1, \dots, x_{2g - 1}) = 1$ in $\Z[t]$. 
		 
		 Let $v \in \Z[t]^{2g}$ be the vector $[0, x_1, \dots, x_{2g - 1}]^\intercal$. Since $v_1$ is in the kernel of $A'$, we have
		 \begin{equation}
			\label{eq:f}
		 	Av = [f(t),0, \dots, 0]^\intercal
		 \end{equation}
		 for some $f(t) \in \Z[t]$. Let $\text{coeff}(f) \subset \Z$ be the set of coefficients of $f$.
		 
		 \begin{claim}
		 	\label{clm:one}
		 	The polynomial $f$ satisfies $\gcd(\text{coeff}(f)) = 1$. In other words, $f$ is primitive.
		 \end{claim}
		 	
		 \begin{proof}[Proof of \Cref{clm:one}]
		 	Suppose $f(t) = d \cdot g(t)$ for some $d \in \Z$, $d \geq 1$, and $g \in \Z[t]$. Let $\bar{X}$ be the universal abelian cover of $S^3 - \nu(K)$, and recall that $A$ presents $H_1(\bar{X})$ as a module over $\Z[t,t^{-1}]$, where $t$ generates the infinite cyclic group of deck transformations. Let $a \in H_1(\bar{X})$ be the first generator corresponding to this presentation; then by (\ref{eq:f}) we have
		 	$$
		 	0 = f(t) \cdot a = d \cdot g(t) \cdot a.
		 	$$
		 	Therefore the element $g(t) \cdot a \in H_1(\bar{X})$ is a torsion homology class, forgetting the $\langle t \rangle $-module structure, and since $H_1(\bar{X})$ is torsion-free (see the remark following the proof of the claim) we have
		 	\begin{equation}
		 		\label{eq:g}
		 		0 = g(t) \cdot a.
		 	\end{equation}
	 		Now (\ref{eq:g}) holds if and only if the equation
	 		\begin{equation}
				\label{eq:w}
		 		Aw = [g(t),0, \dots, 0]^\intercal
		 	\end{equation}
	 		has a solution $w \in \Z[t,t^{-1}]^{2g}$. The vector $(1/d)v \in \Q[t]^{2g}$ is a solution to (\ref{eq:w}), and since $A$ is invertible over $\Q(t)$ this is the unique solution. We conclude that $(1/d)v \in \Z[t]^{2g}$, so $d = 1$ by the construction of $v$.
		 \end{proof}
	 
	 	\begin{rmk}
	 		\label{rmk:tf}
	 		Here is one way to see that $H_1(\bar{X})$ is torsion-free: as in the proof of \Cref{lem:two}, let $X = S^3 - \nu(K)$, let $\mathring{S} = S \cap X$, and let $Y = X - \nu(\mathring{S})$. Let $\{Y_i\}_{i \in \Z}$ be infinitely many copies of $Y$, and let $\mathring{S}^+_i$ and $\mathring{S}^-_i$ be the two components of $\partial \nu(\mathring{S})$ in $\partial Y_i$. Then $\bar{X}$ can be constructed by gluing $\mathring{S}^+_i$ to $\mathring{S}^-_{i + 1}$ for all $i \in \Z$ \cite{rol76}. For $m \geq 0$, let $Z_m \subset \bar{X}$ be the subset
	 		$$
	 		Z_m = Y_{-m} \cup Y_{-m + 1} \cup \cdots \cup Y_{m - 1} \cup Y_m.
	 		$$
	 		Since any singular chain in $\bar{X}$ is supported in a compact set, to show that $H_1(\bar{X})$ is torsion-free it suffices to show that $H_1(Z_m)$ is torsion-free for all $m$. A Mayer-Vietoris calculation shows that for any $m$ $H_1(Z_m) \cong \Z^{2g}$, so $H_1(\bar{X})$ is torsion-free.
	 	\end{rmk}
	 
	 	\begin{claim}
	 		\label{clm:two}
	 		If $f(-1) = \pm 1$, then the crossing change is nugatory.
	 	\end{claim}
	 	
	 	\begin{proof}[Proof of \Cref{clm:two}]
	 		When $t = -1$, the Alexander matrix $A$ becomes $V + V^\intercal$. Since $v \in \Z[t]^{2g}$, $v(-1) \in \Z^{2g}$, and if $f(-1) = \pm 1$ then (\ref{eq:f}) yields
	 		$$
	 		(V + V^\intercal) v(-1) = [\pm 1, 0, \dots, 0]^\intercal.
	 		$$
	 		The claim then follows from \Cref{cor:helpful} and \Cref{thm:lm}.
	 	\end{proof}
 	
 		We will show that $f(t)$ and $f(t^{-1})$ are distinct factors of $\Delta_K$. Let $\hat{\mathcal{A}} = \Q[t^{1/2}, t^{-1/2}]$, and let $\mathcal{A} \subset \hat{\mathcal{A}}$ be the subring $\Q[t, t^{-1}]$. Given $a \in \hat{\mathcal{A}}$, let $\bar{a}$ denote the image of $a$ under the involution determined by $t^{1/2} \mapsto -t^{-1/2}$; on the subring $\mathcal{A}$ this restricts to the involution $t \mapsto t^{-1}$. If $M = \{a_{ij}\}$ is a matrix or vector with elements in $\hat{\mathcal{A}}$ then let $\bar{M} = \{\bar{a}_{ij}\}$, and define
 		$$
 		M^* = \bar{M}^\intercal.
 		$$
		 
		 \begin{claim}
		 	\label{clm:three}
		 	We have
			$$
		 		v^*A = [-tf(t^{-1}),0,\dots, 0].
			$$
		 \end{claim}
	 
	 	\begin{proof}[Proof of \Cref{clm:three}]
	 		Let $A_\text{sym}$ be the matrix
	 		$$
	 		A_\text{sym} = t^{-1/2}A = t^{-1/2}V -t^{1/2}V^\intercal \in M_{2g}(\hat{\mathcal{A}}),
	 		$$
	 		and observe that $A_\text{sym}^* = A_\text{sym}$. We therefore calculate
	 		\begin{align*}
	 		v^*A &= t^{1/2}(v^*A_\text{sym}) = t^{1/2}(v^*A_\text{sym}^*) = t^{1/2}\overline{(v^\intercal A_\text{sym}^\intercal)} \\
	 		&= t^{1/2}(A_\text{sym}v)^* = -t(Av)^* = [-tf(t^{-1}),0,\dots, 0].
	 		\end{align*}
	 	\end{proof}
		 
		 By construction the elements of $v_1 = [x_1, \dots, x_{2g - 1}]$ satisfy $\gcd(x_1, \dots, x_{2g - 1}) = 1$ in $\Z[t, t^{-1}]$, and this also holds in $\mathcal{A}$ by Gauss's lemma. Since $\mathcal{A}$ is a principal ideal domain \cite[Page 117]{mil68}, it follows that $v_1$ extends to a basis $v_1, \dots, v_{2g - 1}$ of the free $\mathcal{A}$-module $\mathcal{A}^{2g - 1}$. Let $B'$ to be the $(2g - 1)$-by-$(2g - 1)$ matrix whose $i$th column vector is $v_i$, and modify the $v_i$ if necessary so that $\det(B') = 1$. Let $B$ be the $2g$-by-$2g$ block diagonal matrix $[1] \oplus B'$, and define
		 $$
		 C = B^*AB.
		 $$
		 Then
		 $$
		 \det(C) = \det(A) = \Delta_K.
		 $$
		 
		 Let $b_1, \dots, b_{2g}$ be the standard basis (column) vectors of $\mathcal{A}^{2g}$. Using (\ref{eq:f}), \Cref{clm:three} and the definition of $B$, we compute
		 $$
		 Cb_2 = B^*Av = B^* [f(t), 0, \dots, 0]^\intercal =  [f(t), 0, \dots, 0]^\intercal
		 $$
		 and
		 $$
		 b_2^\intercal C = v^*AB = [-tf(t^{-1}),0,\dots, 0]B = [-tf(t^{-1}),0,\dots, 0],
		 $$
		 and we conclude that $C$ has the form
		$$
			C = \left[
			\begin{array}{c c c c c}
				c_1 & f(t) & c_2 & \cdots & c_{2g - 1} \\
				-tf(t^{-1}) & 0 & 0 & \cdots & 0 \\
				c_{2g} &0 & \multicolumn{3}{c}{\multirow{3}{*}{\LARGE $C'$}} \\
				\vdots & \vdots & \multicolumn{3}{c}{} \\
				c_{4g - 3} & 0 & \multicolumn{3}{c}{}
			\end{array}
			\right]
		$$
		for some constants $c_1, \dots, c_{4g - 3} \in \mathcal{A}$ and $(2g - 2)$-by-$(2g - 2)$ submatrix $C'$. Therefore, up to multiplication by a unit of $\Z[t,t^{-1}]$,
		\begin{equation}
			\label{eq:delta}
			\Delta_K = \det(C) = f(t) \cdot f(t^{-1}) \cdot \det(C').
		\end{equation}
		Since $\Delta_K$ and $f$ belong to the ring $\Z[t,t^{-1}]$ and $f$ is primitive by \Cref{clm:one}, $\det(C')$ is in $\Z[t,t^{-1}]$ as well. It follows that (\ref{eq:delta}) gives a factorization of $\Delta_K$ in $\Z[t,t^{-1}]$. The proof is completed by observing that, if the crossing change is non-nugatory, then $f(-1) \neq \pm 1$ by \Cref{clm:two}. $\qed$
		
	\section{Knots without Cosmetic Crossings}
	\label{sec:example}
	
	\subsection{A family of pretzel knots}
	\label{sec:family}
	
	In this section, we use \Cref{thm:main} to prove the cosmetic crossing conjecture for an infinite family of knots.
	
	\begin{named_thm}{\protect\Cref{thm:example}}
		Let $p_1$, $p_2$, $p_3$, $p_4$ and $q$ be integers satisfying:
		\begin{enumerate}[label=(\roman*)]
			\item $p_i \geq 1$ for all $i$, and $q > \min(p_1, \dots, p_4)$.
			\item $p_i \equiv 1$ modulo $4$ for all $i$, and $q \equiv 3$ modulo $4$.
		\end{enumerate}
		Then the pretzel knot $K = P(p_1,p_2,p_3,p_4,-q)$ does not admit a non-nugatory, cosmetic generalized crossing change.
	\end{named_thm}

	In \Cref{thm:example}, we've chosen the parameters of the knot $K$ so that our \Cref{thm:main} is necessary to prove the result. In particular, it is not difficult to show that:
	\begin{itemize}
		\item $K$ is not fibered unless $p_1 = \dots = p_4 = 1$ and $q = 3$, since $\Delta_K$ is not monic (see (\ref{eq:seif}) below), so we can't apply \cite{kal12}.
		\item $K$ has genus $g(K) > 1$, since $\Delta_K$ has degree four (see (\ref{eq:mod})), so we can't apply \cite{ito222}.
		\item $K$ is not a special alternating knot, since $K$ satisfies $\sigma(K) < |2g(K)|$, where $\sigma$ denotes the signature, so we can't apply \cite{bon23_3}.
	\end{itemize}
	The knot $K$ also appears to have bridge number higher than two for most choices of $p_1, \dots, p_4$ and $q$, though this is difficult to quantify rigorously.
	
	Additionally, Lidman and Moore prove that if $\Sigma(K)$ is an $L$-space, then $K$ does not admit a non-nugatory cosmetic crossing change if each of the invariant factors of the finite abelian group $H_1(\Sigma(K))$ is square-free \cite[Theorem 2]{lm17}. For many choices of $p_1, \dots, p_4$ and $q$, this result, \cite[Corollary 3.3]{bon23_3}, or \cite[Corollary 1.10]{ito223} suffices to prove the cosmetic crossing conjecture for $K$. For infinitely many other choices, however, \Cref{thm:main} is needed. For example, if $K = P(1,5,1,1,-q)$ for some $q > 1$ congruent to $3$ modulo $4$, then $H_1(K)$ is cyclic of order $16q - 5$. If $q = 100m + 55$ for some $m \geq 0$, then $25$ is a factor of $|H_1(K)|$, and in this case none of the aforementioned results can be applied. It also seems possible that \Cref{thm:main} is necessary to prove \Cref{thm:example} for generic choices of $p_1, \dots, p_4$ and $q$, but we will not attempt to prove this here.

	\begin{proof}[Proof of \protect\Cref{thm:example}]
		Condition (i) ensures the knot $K$ is quasi-alternating by \cite[Theorem 3.2]{ck09-2}, so has branched double cover an $L$-space \cite{os05}. Therefore, by \Cref{thm:main}, it suffices to show $\Delta_K$ has no factor of the form $f(t)f(t^{-1})$.
		
		The bounded checkerboard surface in the standard pretzel diagram of $K$ is orientable, with Seifert matrix given by
		\begin{equation}
			\label{eq:seif}
			V = \begin{bmatrix}
				(p_1 + p_2)/2 & -p_2 & 0 & 0 \\
				0 & (p_2 + p_3)/2 & -p_3 & 0 \\
				0 & 0 & (p_3 + p_4)/2 & -p_4 \\
				0 & 0 & 0 & (p_4 - q)/2
			\end{bmatrix}.
		\end{equation}
		By (ii) the matrix $V$ reduces modulo $2$ to
		$$
		\begin{bmatrix}
			1 & 1 & 0 & 0 \\
			0 & 1 & 1 & 0 \\
			0 & 0 & 1 & 1 \\
			0 & 0 & 0 & 1
		\end{bmatrix},
		$$
		so the Alexander polynomial of $K$ satisfies
		\begin{equation}
			\label{eq:mod}
			\Delta_K(t) = \det(V - tV^\intercal) \equiv 1 + t + t^2 + t^3 + t^4 \mod 2.
		\end{equation}
		Let $\psi(t) \in \F_2[t]$ be the polynomial on the right side of (\ref{eq:mod}); then it is sufficient to show that $\psi$ is irreducible. Certainly $\psi$ has no linear factors, since $\psi(0) = \psi(1) = 1$. Consequently, if $\psi$ factors non-trivially then $\psi$ is a product of two irreducible quadratic polynomials. There is only one such polynomial, $1 + t + t^2$, and
		$$
		(1 + t + t^2)^2 = 1 + t^2 + t^4 \neq \psi(t). 
		$$
		Thus $\psi$ is irreducible, and $\Delta_K$ is as well.
	\end{proof}

	\subsection{Knots with low crossing number}
	\label{sec:eleven}
	
	The cosmetic crossing conjecture is known to be true for all but the following four knots with ten or fewer crossings: $10_{87}$, $10_{98}$, $10_{129}$, and $10_{147}$ \cite[Corollary 1.11]{ito223}\cite{bon23_3}. \Cref{thm:main} does not prove the conjecture for any of these knots, but does verify it for several new prime alternating knots with eleven crossings.\footnote{\protect\Cref{thm:main} also proves the cosmetic crossing conjecture for some non-alternating knots with eleven crossings, but we focus on alternating knots for simplicity.}
	
	There are 367 prime alternating knots with eleven crossings. Of these, 45 knots $K$ satisfy the following four conditions:
	\begin{itemize}
		\item $g(K) > 1$.
		\item The bridge index of $K$ is greater than two.
		\item $K$ is not fibered.
		\item $\det(K)$ is not square-free.
	\end{itemize}
	If $K$ is not in this set of 45, then \Cref{conj:ccc} holds for $K$ by one of the results listed in the introduction \cite{ito222,tor99,kal12,lm17}. Additionally, the conjecture can be proven for 15 of the 45 knots as follows: seven satisfy Lidman and Moore's stronger obstruction \cite[Theorem 2]{lm17}, two satisfy an obstruction of Ito \cite[Corollary 1.10]{ito223}, and a further six are special alternating \cite{bon23_3}. This leaves 30 knots for which the cosmetic crossing conjecture was previously open, and our \Cref{thm:main} proves the conjecture for twelve of these.
	
	In summary:
		
	\begin{thm}
		The cosmetic crossing conjecture holds for all prime alternating knots with eleven crossings, with eighteen possible exceptions:
		\begin{align*}
		&11a_{6}, 11a_{36}, 11a_{38}, 11a_{58}, 11a_{67}, 11a_{87}, 11a_{102}, 11a_{103}, 11a_{104}, 11a_{115}, \\ 
		&11a_{132}, 11a_{165}, 11a_{168}, 11a_{169}, 11a_{199}, 11a_{201}, 11a_{283} \text{, and } 11a_{352}.
		\end{align*}
	\end{thm}
	
	This computation was done with the help of KnotInfo \cite{knotinfo}.
	
	\begin{rmk}
		The most general version of Ito's obstruction \cite[Theorem 1.9]{ito223} would prove the cosmetic crossing conjecture for the knots $11a_{38}$, $11a_{102}$, and $11a_{199}$, included in the above list, if one could show that the branched double cover of the relevant knot can be obtained by Dehn surgery on a knot in $S^3$. For information on the problem of constructing branched double covers of alternating links via knot surgery, see \cite{mccoy15}.
	\end{rmk}
		
	\bibliography{main_bib}{}
	\bibliographystyle{amsplain}
	
\end{document}